\documentclass[12pt,reqno]{article}

\usepackage{amssymb,amscd,amsmath,amsthm}
\numberwithin{equation}{section}
\usepackage{graphicx}

\newtheorem{theorem}{Theorem}[section]
\newtheorem{proposition}[theorem]{Proposition}
\newtheorem{lemma}[theorem]{Lemma}

\newtheorem{definition}[theorem]{Definition}
\newtheorem{remark}[theorem]{Remark}

\def\bn{{\mathbb N}}

\def\br{{\mathbb R}}

\def\a{\alpha}

\def\<{\langle}
\def\>{\rangle}

\def\1{\mathbf{1}}

\def\xb{{\mathbf{x}}}
\def\yb{{\mathbf{y}}}
\def\l{\lambda}

\begin{document}

\begin{center}
{\Large {\bf On unification of the strong convergence theorems for a
finite family of total asymptotically nonexpansive mappings in
Banach spaces}}\\[1cm]

{\sc Farrukh Mukhamedov}


{\it Department of Computational \& Theoretical Sciences \\
Faculty of Sciences, International Islamic University Malaysia\\
P.O. Box, 141, 25710, Kuantan, Pahang, Malaysia\\
 e-mail: {\tt far75m@yandex.ru}, \ {\tt farrukh\_m@iium.edu.my}}\\[5mm]

{\sc Mansoor Saburov}\\[3mm]

{\it Department of Computational \& Theoretical Sciences \\
Faculty of Science, International Islamic University Malaysia\\
P.O. Box, 141, 25710, Kuantan, Pahang, Malaysia\\
e-mail: {\tt msaburov@gmail.com}}
\end{center}

\begin{abstract}
In this paper, we unify all know iterative methods by introducing a
new explicit iterative scheme for approximation of common fixed
points of finite families of total asymptotically $I$-nonexpansive
mappings. Note that such a scheme contains as a particular case of
the method introduced in [C.E. Chidume, E.U. Ofoedu, \textit{Inter.
J. Math. \& Math. Sci.} \textbf{2009}(2009) Article ID 615107, 17p].
We construct examples of total asymptotically nonexpansive mappings
which are not asymptotically nonexpansive. Note that no such kind of
examples were known in the literature. We prove the strong
convergence theorems for such iterative process to a common fixed
point of the finite family of total asymptotically $I-$nonexpansive
and total asymptotically nonexpansive mappings, defined on a
nonempty closed convex subset of uniformly convex Banach spaces.
Moreover, our results extend and unify all known results.   \vskip
0.3cm \noindent {\it
Mathematics Subject Classification}: 46N10; 46B20; 47H09; 47H10
\end{abstract}

\section{Introduction}
Let $K$ be a nonempty subset of a real normed linear space $X$ and
$T:K\to K$ be a mapping.
 Denote by $F(T)$ the set of fixed points of $T$, that is, $F(T) =\{x\in K: Tx = x\}$. Throughout this paper,
 we always assume that $X$ is a real Banach space and $F(T)\neq\emptyset$. Now let us recall some known definitions
\begin{definition}
A mapping $T:K\to K$ is said to be:
\begin{itemize}
  \item[(i)] nonexpansive, if $\|Tx-Ty\|\le\|x-y\|$
  for all $x,y\in K$;
  \item[(ii)]  asymptotically nonexpansive,
  if there exists a sequence $\{\lambda_{n}\}\subset[1,\infty)$ with
  $\lim\limits_{n\to\infty}\lambda_{n}=1$ such that
  $\|T^nx-T^ny\|\le\lambda_n\|x-y\|$ for all $x,y\in K$ and $n\in\bn$;
  \item[(iii)] asymptotically nonexpansive in the intermediate sense,
  if it is continuous and the following inequality holds
  \begin{eqnarray}\label{asympnonexpintersense}
  \limsup\limits_{n\to\infty}\sup\limits_{x,y\in
  K}(\|T^nx-T^ny\|-\|x-y\|)\le0.
  \end{eqnarray}
  \end{itemize}
\end{definition}

\begin{remark}
Observe that if we define
\begin{eqnarray}\label{anandsigman}
a_n:=\sup\limits_{x,y\in K}(\|T^nx-T^ny\|-\|x-y\|), \quad\quad
\sigma_n:=\max\{0,a_n\},
\end{eqnarray}
then $\sigma_n\to 0$ as $n\to\infty$ and
\eqref{asympnonexpintersense} reduces to
\begin{eqnarray}\label{asypmnonexpintersenseintermsofsigma}
\|T^nx-T^ny\|\le\|x-y\|+\sigma_n, \quad\quad  \forall x,y\in K,
n\ge1.
\end{eqnarray}
\end{remark}

In \cite{[Browder65]}-\cite{[Browder67]} Browder  studied the
iterative construction for fixed points of nonexpansive mappings on
closed and convex subsets of a Hilbert space. Note that for the past
30 years or so, the study of the iterative processes for the
approximation of fixed points of nonexpansive mappings and fixed
points of some of their generalizations have been flourishing areas
of research for many mathematicians (see for more details
\cite{GK2},\cite{Chidumi}).

The class of asymptotically nonexpansive mappings was introduced by
Goebel and Kirk \cite{[GeobelKrik]} as a generalization of the class
of nonexpansive mappings. They proved that if $K$ is a nonempty
closed convex bounded subset of a uniformly convex real Banach space
and $T$ is an asymptotically nonexpansive self-mapping of $K,$ then
$T$ has a fixed point.

The class of mappings which are asymptotically nonexpansive in the
intermediate sense was introduced by Bruck et al.
\cite{[BruckKuczumowReich]}. It is known \cite{[Kirk]} that if $K$
is a nonempty closed convex bounded subset of a uniformly convex
Banach space $X$ and $T:K\to K$ is an asymptotically nonexpansive
mapping in the intermediate sense, then $T$ has a fixed point. It is
worth mentioning that the class of mappings which are asymptotically
nonexpansive in the intermediate sense contains properly the class
of asymptotically nonexpansive mappings (see, e.g., \cite{[KimKim]})

The iterative approximation problems for nonexpansive mapping,
asymptotically nonexpansive mapping and asymptotically nonexpansive
mapping in the intermediate sense were studied extensively in
\cite{[GeobelKrik]}, \cite{[Kirk]}, \cite{[KimKim]},
\cite{[BruckKuczumowReich]}, \cite{[Bruck]},
\cite{liu},\cite{[Wittmann]}, \cite{R}, \cite{[Gornicki]},
\cite{[Shu]},\cite{ST},
\cite{[TanXu]},\cite{CAC},\cite{CCT},\cite{CKZ},\cite{CZG}.

There are many different types of concepts which generalize a notion
of nonexpansive mapping. One of such concepts is a total
asymptotically nonexpansive mapping (\cite{[AlberChidumeZegeye]})
and second one is an asymptotically $I$-nonexpansive mapping
(\cite{Shah}). Let us recall some notions.

\begin{definition}\label{deftotalasympexpan}
Let $K$ be a nonempty closed subset of a real normed linear space
$X.$ $T:K\to K$ is called a total asymptotically nonexpansive
mapping if there exist nonnegative real sequence $\{\mu_n\}$ and
$\{\lambda_n\}$ with $\mu_n,\lambda_n\to 0$ as $n\to\infty$ and
strictly increasing continuous function
$\phi:{\mathbb{R}}^{+}\to{\mathbb{R}}^{+}$ with $\phi(0)=0$ such
that for all $x,y\in K,$
\begin{eqnarray}\label{inequalityfortotal}
\|T^nx-T^ny\|\le\|x-y\|+\mu_n\phi(\|x-y\|)+\lambda_n,\quad\quad
n\ge1.
\end{eqnarray}
\end{definition}

\begin{remark}
If $\phi(\xi)=\xi,$ then \eqref{inequalityfortotal} reduces to
\begin{eqnarray}\label{inequalityphi=xi}
\|T^nx-T^ny\|\le(1+\mu_n)\|x-y\|+\lambda_n,\quad\quad n\ge1.
\end{eqnarray}
In addition, if $\lambda_n=0$ for all $n\ge1,$ then total
asymptotical nonexpansive mappings coincide with asymptotically
nonexpansive mappings. If $\mu_n=0$ and $\lambda_n=0$ for all
$n\ge1,$ we obtain from \eqref{inequalityphi=xi} the class of
mappings that includes the class of nonexpansive mappings. If
$\mu_n=0$ and $\lambda_n=\sigma_n=\max\{0,a_n\},$ where
$a_n:=\sup\limits_{x,y\in K}(\|T^nx-T^ny\|-\|x-y\|)$ for all
$n\ge1,$ then \eqref{inequalityphi=xi} reduces to
\eqref{asypmnonexpintersenseintermsofsigma} which has been studied
as mappings asymptotically nonexpansive in the intermediate sense.
\end{remark}

The idea of the definition of a total asymptotically nonexpansive
mappings is that to unify various definitions of classes of mappings
associated with the class of asymptotically nonexpansive mappings
and to prove a general convergence theorems applicable to all these
classes of nonlinear mappings.

Alber et.al. \cite{[AlberChidumeZegeye]} studied methods of
approximation of fixed points of total asymptotically nonexpansive
mappings. Chidume et.al. \cite{[ChidumeOfoedu]} introduced an
iterative scheme for approximation of a common fixed point of a
finite family of total asymptotically nonexpansive mappings in
Banach spaces. Recently, Chidume et.al. \cite{[ChidumeOfoedu2009]}
constructed a new iterative sequence much simpler that other types
of approximation of common fixed points of finite families of total
asymptotically nonexpansive mappings.

On the other hand, in \cite{Shah} an asymptotically $I$-nonexpansive
mapping was introduced.

\begin{definition}\label{defasympInonexpansive}
Let $T:K\to K$, $I:K\to K$ be two mappings of a nonempty subset $K$
of a real normed linear space $X$. Then $T$ is said to be:
\begin{itemize}
\item[(i)] {\it $I-$nonexpansive}, if $\|Tx-Ty\|\le\|Ix-Iy\|$ for all $x,y\in K$;
\item[(i)] {\it asymptotically $I-$nonexpansive}, if there exists a sequence $\{\lambda_{n}\}\subset[1,\infty)$ with   $\lim\limits_{n\to\infty}\lambda_{n}=1$ such that
  $\|T^nx-T^ny\|\le\lambda_n\|I^nx-I^ny\|$ for all $x,y\in K$ and $n\ge 1$;
\end{itemize}
\end{definition}

Best approximation properties of $I$-nonexpansive mappings were
investigated in \cite{Shah,CHP}. In \cite{RT} strong convergence of
Mann iterations of $I$-nonexpansive mapping has been proved.  In
\cite{KKJ} the weak convergence of three-step Noor iterative scheme
for an $I$-nonexpansive mappping in a Banach space has been
established. In \cite{TG} the weakly convergence theorem for
asymptotically $I$-nonexpansive mapping defined in Hilbert space was
proved. Recently, in \cite{Gu,T,T1} the weak and strong convergence
of explicit and implicit iteration process to a common fixed point
of a finite family of asymptotically $I$-nonexpansive mappings have
been studied.

In this paper, we introduce a new type of concept of a
generalization of nonexpansive mapping's nation,  which is a
combination of Definitions \ref{deftotalasympexpan} and
\ref{defasympInonexpansive}.

\begin{definition}
Let $T:K\to K$, $I:K\to K$ be two mappings of a nonempty subset $K$
of a real normed linear space $X$.  Then $T$ is said to be a total
asymptotically $I$-nonexpansive mapping if there exist nonnegative
real sequences $\{\mu_n\}$ and $\{\lambda_n\}$ with
$\mu_n,\lambda_n\to 0$ as $n\to\infty$ and the strictly increasing
continuous function $\phi:{\mathbb{R}}^{+}\to{\mathbb{R}}^{+}$ with
$\phi(0)=0$ such that for all $x,y\in K,$
\begin{eqnarray}\label{inequalityfortotalInonexpan}
\|T^nx-T^ny\|\le\|I^nx-I^ny\|+\mu_n\phi(\|I^nx-I^ny\|)+\lambda_n,\quad\quad
n\ge1.
\end{eqnarray}
\end{definition}

Now let us provide an examples of a total asymptotically
$I$-nonexpansive mapping, which is not asymptotically nonexpansive
mapping.

{\bf Example 1.} Let us consider the space $\ell_1$, and let
$B_1=\{\xb\in\ell_1: \ \|\xb\|_1\leq 1\}$. Define a nonlinear
operator $T:\ell_1\to\ell_1$ by
\begin{equation}\label{T}
T_\a(x_1,x_2\dots,x_n,\dots)=\big(0,\a\sqrt{|x_1|},\a x_2,\dots,\a
x_n,\dots\big), \ \ \ \a\in(0,1).
\end{equation}
Let $\|\xb\|_1\leq 1$, then from
\begin{eqnarray*}
\|T_\a(\xb)\|_1&=&\a\bigg(\|\xb\|_1-|x_1|+\sqrt{|x_1|}\bigg)\\[2mm]
&\leq&\a\bigg(1-|x_1|+\sqrt{|x_1|}\bigg)\leq 1
\end{eqnarray*}
one gets $T(B_1)\subset B_1$.

One can find that
\begin{equation*}
T^k_\a(x_1,x_2\dots,x_n,\dots)=\big(\underbrace{0,\dots,0}_k,\a^k\sqrt{|x_1|},
\a^k x_2,\dots,\a^k x_n,\dots\big).
\end{equation*}
Hence,
\begin{eqnarray}\label{Tk}
\|T^k_\a(\xb)-T^k_\a(\yb)\|_1=\a^k\bigg(\|\xb-\yb\|_1+
\big|\sqrt{|x_1|}-\sqrt{|y_1|}\big|-|x_1-y_1|\bigg).
\end{eqnarray}
From $\xb,\yb\in B_1$ we have
\begin{eqnarray}\label{Tk1}
\big|\sqrt{|x_1|}-\sqrt{|y_1|}\big|\leq
\sqrt{\big||x_1|-|y_1|\big|}\leq \sqrt{\|\xb-\yb\|_1}.
\end{eqnarray}
So, it follows from \eqref{Tk},\eqref{Tk} that
\begin{eqnarray}\label{Tk2}
\|T^k_\a(\xb)-T^k_\a(\yb)\|_1\leq\a^k\bigg(\|\xb-\yb\|_1+
\sqrt{\|\xb-\yb\|_1}\bigg) \ \ \ \textrm{for all} \ \xb,\yb\in B_1,
\ k\in\bn.
\end{eqnarray}

Now consider a new Banach space $\br\times \ell_1$ with a norm
$\|\mathbb{X}\|=|x|+\|\xb\|_1$, where $\mathbb{X}=(x,\xb)$ and
define a new mapping $S: \br\times \ell_1\to\br\times \ell_1$ by
$$
S(x,\xb)=(x,T_\a(\xb)).
$$
Let $K=[0,1]\times B_1$. Then it is clear that $S(K)\subset K$. One
can see that $S^k(x,\xb)=(x,T^k_\a(\xb))$. Therefore, using
\eqref{Tk2} we obtain
\begin{eqnarray}\label{Tk2}
\|S^k(\mathbb{X})-S^k(\mathbb{Y})\|&=&|x-y|+\|T^k_\a(\xb)-T^k_\a(\yb)\|_1\nonumber\\[2mm]
&\leq&|x-y|+\a^k\bigg(\|\xb-\yb\|_1+
\sqrt{\|\xb-\yb\|_1}\bigg)\nonumber\\[2mm]
&\leq&\|\mathbb{X}-\mathbb{Y}\|+\a^k\bigg(\|\mathbb{X}-\mathbb{Y}\|+
\sqrt{\|\mathbb{X}-\mathbb{Y}\|}\bigg)
\end{eqnarray}
We let $\phi(t)=t+\sqrt{t}$ and $\mu_k=\a^k$. It is clear that
$\phi(0)=0$ and $\phi$ is strictly increasing, and moreover
\eqref{Tk2} implies
$$
\|S^k(\mathbb{X})-S^k(\mathbb{X})\|\leq
\|I\mathbb{X}-I\mathbb{Y}\|+\mu_k\phi(\|I\mathbb{X}-I\mathbb{Y}\|)
$$
that $S$ is a totaly asymptotically  $I$-nonexpansive mapping. Here
$I$ is the identity mapping of $\br\times\ell_1$.

Now we are going to show that $S$ is not asymptotically
nonexpansive. Namely, we will establish that for any sequence of
positive numbers $\{\l_n\}$ with $\l_n\to 0$ and any $k\in\bn$ one
can find $\mathbb{X}_0,\mathbb{Y}_0$ such that
$$
\|S^k(\mathbb{X}_0)-S^k(\mathbb{Y}_0)\|>(1+\l_k)\|\mathbb{X}_0-\mathbb{Y}_0\|.
$$
In fact, choose $\mathbb{X}_0$, $\mathbb{Y}_0$ as follows:
$$
\mathbb{X}_0=(0,\xb_0), \ \ \mathbb{Y}_0=(0,\yb_0),
$$
where
$$
\xb_0=(x_0,0,\dots,0,\dots), \ \
\yb_0=(\frac{x_0}{4},0,\dots,0,\dots)
$$
and
\begin{equation}\label{x0}
 0<x_0<\frac{4\a^{2k}}{9(1+\l_k)^2}.
\end{equation}

From \eqref{Tk} one finds
\begin{eqnarray*}
\|S^k(\mathbb{X}_0)-S^k(\mathbb{Y}_0)\|&=&
\a^k\big|\sqrt{x_0}-\frac{\sqrt{x_0}}{2}\big|=\a^k\frac{\sqrt{x_0}}{2},\\
\|\mathbb{X}_0-\mathbb{Y}_0\|&=&\frac{3x_0}{4}.
\end{eqnarray*}
The last equalities with \eqref{x0} imply that
\begin{eqnarray*}
\frac{\|S^k(\mathbb{X}_0)-S^k(\mathbb{Y}_0)\|}{\|\mathbb{X}_0-\mathbb{Y}_0\|}=
\frac{2\a^k}{3\sqrt{x_0}}>1+\l_k.
\end{eqnarray*}
This yields the required assertion. Note that $S$ has infinitely
many fixed points in $K$, i.e. $Fix(S)=\{(x,\mathbf{0}):\ \
x\in[0,1]\}$.

{\sc Example 2.} Let us consider the Banach space $\br\times \ell_1$
defined as before, and $f$ be a mapping of a segment $C\subset\br$
to itself, i.e. $f:C\to C$ with $f(0)=0$ and
$$
|f^n(x)-f^n(y)|\leq |x-y|+c_n, \ \ c_n>0, \ \ n\in\bn
$$
where $c_n\to 0$. Note that such kind of functions do exist. One can
take (see for more details \cite{[KimKim]}) $C=[-1/\pi,1/\pi]$ and
$$
f_\kappa(x)= \left\{
\begin{array}{ll}
\kappa x\sin\frac{1}{x}, \ \ x\neq 0\\[2mm]
0, \ \ x=0.
\end{array}
\right. \ \ \kappa\in(0,1).
$$

 Define a new mapping $S_f: C\times B_1\to C\times B_1$ by
$$
S_f(x,\xb)=(f(x),T_\a(\xb)),
$$
here $T$ is defined as above (see \eqref{T}). Using the same
argument as above Example 1, we can establish that
$$
\|S^k_f(\mathbb{X})-S^k_f(\mathbb{X})\|\leq
\|\mathbb{X}-\mathbb{Y}\|+\mu_k\phi(\|\mathbb{X}-\mathbb{Y}\|)+c_k \
\ \forall k\in\bn.
$$
Moreover, such a mapping is not asymptotically nonexpansive. Note
that the mapping $S_{f_\kappa}$ with the function $f_\kappa$ has a
unique fixed point in $C\times B_1$.

{\bf Remark.} To the best our knowledge, we should stress that the
constructed examples are currently only unique examples of totaly
asymptotically nonxpansive mappings which are not asymptotically
nonxpansive. Before, no such examples were known in the literature.\\

 Aim of the present paper is unification of all know iterative
methods by introducing a new iterative scheme for approximation of
common fixed points of finite families of total asymptotically
$I$-nonexpansive mappings. Note that such a scheme contains as a
particular case of the method introduced in
\cite{[ChidumeOfoedu2009]}, and allow us to construct more simpler
methods than \cite{[ChidumeOfoedu],[ChidumeOfoedu2009]}.

Namely, let $K$ be a nonempty closed convex subset of a real Banach
space $X$ and $\{T_i\}_{i=1}^m:K\to K$ be a finite family of total
asymptotically $I_i-$nonexpansive mappings, i.e.
\begin{eqnarray}\label{defofTnxTny}
\|T_i^nx-T_i^ny\|\le\|I_i^nx-I_i^ny\|+\mu_{in}\phi_i(\|I_i^nx-I_i^ny\|)+\lambda_{in}
\end{eqnarray} and $\{I_i\}_{i=1}^m:K\to K$ be a  finite family of total asymptotically nonexpansive
mappings, i.e.
\begin{eqnarray}
\|I_i^nx-I_i^ny\|\le\|x-y\|+\tilde{\mu}_{in}\varphi_i(\|x-y\|)+\tilde{\lambda}_{in},
\end{eqnarray} here
$\phi_i, \varphi_i:{\mathbb{R}}^{+}\to{\mathbb{R}}^{+}$ are the
strictly increasing continuous functions with
$\phi_i(0)=\varphi_i(0)=0$ for all $i=\overline{1,m}$ and
$\{\mu_{in}\}_{n=1}^\infty, \{\lambda_{in}\}_{n=1}^\infty,
\{\tilde\mu_{in}\}_{n=1}^\infty,
\{\tilde\lambda_{in}\}_{n=1}^\infty$ are nonnegative real sequences
with $\mu_{in},\lambda_{in},\tilde\mu_{in},\tilde\lambda_{in}\to 0$
as $n\to\infty$ for all $i=\overline{1,m}.$ Then for given sequences
$\{\alpha_{jn}\}_{n=1}^\infty,$ $\{\beta_{jn}\}_{n=1}^\infty$ in
$(0,1),$ where $j=\overline{0,m},$ we shall consider the following
explicit iterative process:
\begin{eqnarray}\label{explicitmap}
\left\{ \begin{array}{ccc}
          x_0\in K,\\
          x_{n+1} =  \alpha_{0n}x_{n}+\sum\limits_{i=1}^m\alpha_{in} T_i^ny_n \\
          y_n =  \beta_{0n}x_n+\sum\limits_{i=1}^m\beta_{in} I_i^nx_n.
        \end{array}\right.
\end{eqnarray}
such that $\sum\limits_{j=0}^m\alpha_{jn}=1$ and $\sum\limits_{j=0}^m\beta_{jn}=1.$

Chidume et.al. \cite{[ChidumeOfoedu2009]} has considered only a
particular case of the explicit iterative process
\eqref{explicitmap}, in which $\{I_i\}_{i=1}^m$ to be taken as the
identity mappings. One of the main results of
\cite{[ChidumeOfoedu2009]} (see Theorem 3.5. p.11) was correct while
the provided proof of that result was wrong. Since, in their proof
they used Lemma \ref{convexxnynwithtn}, but which actually is not
applicable in that situation, since the sequence
$\{t_n\}_{n=1}^\infty$ tends to $0.$ As a counterexample, we can
consider the following one: let $x\in X,$ $\|x\|=d>0,$ and let the
sequences $x_n, y_n, t_n$ be defined as follows
$$x_n=x,\quad y_n=-x, \quad t_n=\frac{1}{n}, \quad \forall n\in
{\mathbb{N}}.$$ It is then clear that
$$\lim\limits_{n\to\infty}\|t_nx_n+(1-t_n)y_n\|=
\|x\|\lim\limits_{n\to\infty}\left|1-\frac{2}{n}\right|=d.$$
However, $$\lim\limits_{n\to\infty}\|x_n-y_n\|=2d>0.$$

In this paper, we shall provide a correct proof of Theorem 3.5 p.11
in \cite{[ChidumeOfoedu2009]}. As we already mentioned Lemma
\ref{convexxnynwithtn} is not applicable the main result of
\cite{[ChidumeOfoedu2009]}. Therefore, we first will generalize
Lemma \ref{convexxnynwithtn} to the case of finite number of
sequences. Such a generalization gives us a possibility to prove the
mentioned result. On other hand, the provided generalization
presents an independent interest as well. Moreover, we extend and
unify the main result of \cite{[ChidumeOfoedu2009]} for a finite
family of total asymptotically $I_i-$nonexpansive mappings
$\{T_i\}_{i=1}^m.$ Namely, we shall prove the strong convergence of
the explicit iterative process \eqref{explicitmap} to a common fixed
point of the finite family of total asymptotically
$I_i-$nonexpansive mappings $\{T_i\}_{i=1}^m$ and the finite family
of total asymptotically nonexpansive mappings $\{I_i\}_{i=1}^m.$
Here, we stress that Lemmas \ref{generofZeidler} and
\ref{generofShu} play a crucial role. All presented results here
extend, generalize, unify and improve the corresponding main results
of \cite{[AlberChidumeZegeye]}, \cite{[ChidumeOfoedu2009]},
\cite{Gu}, \cite{T},\cite{T1},\cite{[Sahu]}, \cite{[ShahzadUdomene],
\cite{MS1}}.

\section{Preliminaries}

Throughout this paper, we always assume that $X$ is a real Banach
space. The following lemmas play an important role in proving our
main results.

\begin{lemma}\cite{[TanXu]}\label{convergean}
Let $\{a_n\}, \{b_n\}, \{c_n\}$ be three sequences of nonnegative
real numbers with $\sum\limits_{n=1}^{\infty}b_n<\infty,$
$\sum\limits_{n=1}^{\infty}c_n<\infty.$ If the following condition
is satisfied
\begin{itemize}
    \item[(i)] $a_{n+1}\le(1+b_n)a_n+c_n, \ \ \ n\ge 1.$
\end{itemize}
then the limit $\lim\limits_{n\to\infty}a_n$ exists.
\end{lemma}

\begin{lemma}\cite{[Zeidler]}\label{convexxnynwitht}
Let $X$ be a uniformly convex Banach space and $t\in(0,1).$ Suppose $\{x_n\},$ $\{y_n\}$ are two sequences in $X$ such that
\begin{eqnarray*}
\lim\limits_{n\to\infty}\|tx_n+(1-t)y_n\|=d, \ \ \
\limsup\limits_{n\to\infty}\|x_n\|\le d, \ \ \
\limsup\limits_{n\to\infty}\|y_n\|\le d,
\end{eqnarray*}
holds some $d\ge 0.$ Then $\lim\limits_{n\to\infty}\|x_n-y_n\|=0.$
\end{lemma}

\begin{lemma}\cite{[Shu]}\label{convexxnynwithtn}
Let $X$ be a uniformly convex Banach space and $b,c$ be two
constants with $0<b<c<1.$ Suppose that $\{t_n\}$ is a sequence in
$[b,c]$ and $\{x_n\},$ $\{y_n\}$ are two sequences in $X$ such that
\begin{eqnarray*}
\lim\limits_{n\to\infty}\|t_nx_n+(1-t_n)y_n\|=d, \ \ \
\limsup\limits_{n\to\infty}\|x_n\|\le d, \ \ \
\limsup\limits_{n\to\infty}\|y_n\|\le d,
\end{eqnarray*}
holds some $d\ge 0.$ Then $\lim\limits_{n\to\infty}\|x_n-y_n\|=0.$
\end{lemma}

\section{Main results}

In this section we shall prove our main results. To formulate ones,
we need some auxiliary results.

First we are going to generalize of Lemmas \ref{convexxnynwitht} and
\ref{convexxnynwithtn} for $m$ number of sequences
$\{z_{in}\}_{n=1}^\infty$ from the uniformly convex Banach space
$X,$ where $i=\overline{1,m}.$

\begin{lemma}\label{generofZeidler}
Let $X$ be a uniformly convex Banach space and $\alpha_i\in(0,1),
i=\overline{1,m}$ be any constants with
$\sum\limits_{i=1}^m\alpha_i=1.$ Suppose $\{z_{in}\}_{n=1}^\infty,
i=\overline{1,m}$ are sequences in $X$ such that
\begin{eqnarray}\label{convexalphaizinwithconst}
\lim\limits_{n\to\infty}\left\|\sum\limits_{i=1}^m\alpha_iz_{in}\right\|=d,
\ \ \ \limsup\limits_{n\to\infty}\|z_{in}\|\le d,\ \ \ \forall
i=\overline{1,m},
\end{eqnarray}
holds some $d\ge 0.$ Then $\lim\limits_{n\to\infty}\|z_{in}\|=d$ and
$\lim\limits_{n\to\infty}\|z_{in}-z_{jn}\|=0$ for any
$i,j=\overline{1,m}.$
\end{lemma}
\begin{proof}
Let us first prove $\lim\limits_{n\to\infty}\|z_{in}\|=d$ for any
$i=\overline{1,m}.$ Indeed, it follows from
\eqref{convexalphaizinwithconst} that
\begin{eqnarray*}
d&=&\lim\limits_{n\to\infty}\left\|\sum\limits_{k=1}^m\alpha_kz_{kn}\right\|
=\liminf\limits_{n\to\infty}\left\|\sum\limits_{k=1}^m\alpha_kz_{kn}\right\|\\
&\le&\liminf\limits_{n\to\infty}\left(\sum\limits_{k=1}^m\alpha_k\|z_{kn}\|\right)
\le\alpha_i\liminf\limits_{n\to\infty}\|z_{in}\|+\sum\limits_{k\neq i}\alpha_k\limsup\limits_{n\to\infty}\|z_{kn}\|\\
&\le&\alpha_i\liminf\limits_{n\to\infty}\|z_{in}\|+(1-\alpha_i)d.
\end{eqnarray*}
We then get that $\liminf\limits_{n\to\infty}\|z_{in}\|\ge d,$ which
means $\lim\limits_{n\to\infty}\|z_{in}\|=d.$

Now we prove the statement
$\lim\limits_{n\to\infty}\|z_{in}-z_{jn}\|=0$ by means of
mathematical induction w.r.t. $m.$ For $m=2,$ the statement
immediately  follows from Lemma \ref{convexxnynwitht}. Assume that
the statement is true, for $m=k-1$. Let us prove for $m=k.$ To do
this, denote
$$t_n=\frac{1}{1-\alpha_k}\sum\limits_{i=1}^{k-1}\alpha_iz_{in}.$$
Since $\frac{1}{1-\alpha_k}\sum\limits_{i=1}^{k-1}\alpha_i=1$ we get
$\limsup\limits_{n\to\infty}\|t_n\|\le d.$ On the other hand, one
has
\begin{eqnarray*}
d&=&\liminf\limits_{n\to\infty}\left\|\sum\limits_{i=1}^k\alpha_iz_{in}\right\|=\liminf\limits_{n\to\infty}\|(1-\alpha_k)t_n+\alpha_kz_{kn}\|\\
&\le& (1-\alpha_k)\liminf\limits_{n\to\infty}\|t_n\|+\alpha_k\limsup\limits_{n\to\infty}\|z_{kn}\|\\
&\le& (1-\alpha_k)\liminf\limits_{n\to\infty}\|t_n\| +\alpha_k d.
\end{eqnarray*}
We then obtain $\liminf\limits_{n\to\infty}\|t_n\|\ge d$ which means
$\lim\limits_{n\to\infty}\|t_n\|=d.$ In this case, according to the
assumption of induction with the sequence $t_n,$  we can conclude
that $\lim\limits_{n\to\infty}\|z_{in}-z_{jn}\|=0,$ if $1\le i,j\le
k-1.$

Since $\lim\limits_{n\to\infty}\|(1-\alpha_k)t_n+\alpha_kz_{kn}\|=d$
due to Lemma \ref{convexxnynwitht} one gets
$$\lim\limits_{n\to\infty}\|t_n-z_{kn}\|=0.$$ If $1\le j\le k-1$
then the following inequality
\begin{eqnarray*}
\|z_{jn}-z_{kn}\|&\le&\|z_{jn}-t_n\|+\|t_n-z_{kn}\|\\
&\le&\frac{1}{1-\alpha_k}\sum\limits_{i=1}^{k-1}\alpha_i\|z_{in}-z_{jn}\|+\|t_n-z_{kn}\|
\end{eqnarray*}
implies that $\lim\limits_{n\to\infty}\|z_{jn}-z_{kn}\|=0.$ This
completes the proof.
\end{proof}

\begin{lemma}\label{generofShu}
Let $X$ be a uniformly convex Banach space and
$\alpha_{*},\alpha^{*}$ be two constants with
$0<\alpha_{*}<\alpha^{*}<1.$ Suppose that
$\{\alpha_{in}\}_{n=1}^{\infty}\subset[\alpha_{*},\alpha^{*}],$
$i=\overline{1,m}$ are any sequences with
$\sum\limits_{i=1}^m\alpha_{in}=1$ for all $n\in \bn.$ Suppose
$\{z_{in}\}_{n=1}^\infty, i=\overline{1,m}$ are sequences in $X$
such that
\begin{eqnarray}\label{convexalphaizinwithsequn}
\lim\limits_{n\to\infty}\left\|\sum\limits_{i=1}^m\alpha_{in}z_{in}\right\|=d,
\ \ \ \limsup\limits_{n\to\infty}\|z_{in}\|\le d, \ \ \forall
i=\overline{1,m},
\end{eqnarray}
holds for some $d\ge 0.$ Then $\lim\limits_{n\to\infty}\|z_{in}\|=d$
and $\lim\limits_{n\to\infty}\|z_{in}-z_{jn}\|=0$ for any
$i,j=\overline{1,m}.$  \end{lemma}
\begin{proof}
Analogously as in the proof of Lemma \ref{generofZeidler}, it is
easy to show that $\lim\limits_{n\to\infty}\|z_{in}\|=d.$ Therefore,
let us prove the statement
$\lim\limits_{n\to\infty}\|z_{in}-z_{jn}\|=0$ for any
$i,j=\overline{1,m}.$ Suppose contrary, i.e., there exist two
numbers $i_0,j_0$ such that
$$
\limsup\limits_{n\to\infty}\|z_{i_0n}-z_{j_0n}\|=\beta_{i_0,j_0}>0.$$
Then, there exists a subsequence
$\{z_{i_0n_k}-z_{j_0n_k}\}_{k=1}^\infty$ of
$\{z_{i_0n}-z_{j_0n}\}_{n=1}^\infty$ such that
$\lim\limits_{k\to\infty}\|z_{i_0n_k}-z_{j_0n_k}\|=\beta_{i_0j_0}.$

Let us consider the subsequences $\{\alpha_{in_k}\}_{k=1}^{\infty}$
of $\{\alpha_{in}\}_{n=1}^{\infty},$ here $i=\overline{1,m}.$ Since
$\{\alpha_{in_k}\}_{k=1}^{\infty}\subset[\alpha_{*},\alpha^{*}]$
there exist a subsequence $\{n_{k_l}\}_{l=1}^\infty$ of
$\{n_{k}\}_{k=1}^\infty$ such that
$\lim\limits_{l\to\infty}\alpha_{in_{k_l}}=\alpha_i$ for all
$i=\overline{1,m}.$ Since $\sum\limits_{i=1}^m\alpha_{in}=1$, for
all $n\in \bn$, one gets $\sum\limits_{i=1}^m\alpha_{i}=1,$ and
$\alpha_i\in[\alpha_{*},\alpha^{*}],$ for all $i=\overline{1,m}.$ We
know that
\begin{eqnarray*}
d&=&\lim\limits_{l\to\infty}\left\|\sum\limits_{i=1}^m\alpha_{in_{k_l}}z_{in_{k_l}}\right\|
=\liminf\limits_{l\to\infty}\left\|\sum\limits_{i=1}^m\left((\alpha_{in_{k_l}}-\alpha_{i})z_{in_{k_l}}+\alpha_{i}z_{in_{k_l}}\right)\right\|\\
&\le&\liminf\limits_{l\to\infty}\left(\sum\limits_{i=1}^m|\alpha_{in_{k_l}}-\alpha_{i}|\|z_{in_{k_l}}\|+
\left\|\sum\limits_{i=1}^m\alpha_{i}z_{in_{k_l}}\right\|\right)\\
&\le&\sum\limits_{i=1}^m\limsup\limits_{l\to\infty}\left(|\alpha_{in_{k_l}}-\alpha_{i}|\|z_{in_{k_l}}\|\right)+
\liminf\limits_{l\to\infty}\left\|\sum\limits_{i=1}^m\alpha_{i}z_{in_{k_l}}\right\|\\
\end{eqnarray*}
It then follows that
$\liminf\limits_{l\to\infty}\left\|\sum\limits_{i=1}^m\alpha_{i}z_{in_{k_l}}\right\|\ge
d.$ On the other hand, we have
$$\limsup\limits_{l\to\infty}\left\|\sum\limits_{i=1}^m\alpha_{i}z_{in_{k_l}}\right\|
\le
\sum\limits_{i=1}^m\alpha_{i}\limsup\limits_{l\to\infty}\left\|z_{in_{k_l}}\right\|\le
d.$$ Therefore,
$\lim\limits_{l\to\infty}\left\|\sum\limits_{i=1}^m\alpha_{i}z_{in_{k_l}}\right\|=d.$
Consequently, Lemma \ref{generofZeidler} implies that
$\lim\limits_{l\to\infty}\|z_{i_0n_{k_l}}-z_{j_0n_{k_l}}\|=0.$
However, it contradicts to
$$\lim\limits_{l\to\infty}\|z_{i_0n_{k_l}}-z_{j_0n_{k_l}}\|=\lim\limits_{k\to\infty}\|z_{i_0n_k}-z_{j_0n_k}\|=\beta_{i_0j_0}>0.$$ This completes the proof.
\end{proof}

\begin{proposition}\label{extimationofTnxTny}
Let $X$ be a real Banach space and $K$ be a nonempty closed convex
subset of $X.$ Let $\{T_i\}_{i=1}^m:K\to K$ be a finite family of
total asymptotically $I_i-$nonexpansive mappings with  sequences
$\{\mu_{in}\}_{n=1}^{\infty}, \{\lambda_{in}\}_{n=1}^{\infty},$
where $i=\overline{1,m},$ and $\{I_i\}_{i=1}^{m}:K\to K$ be a finite
family total asymptotically nonexpansive mappings with sequences
$\{\tilde\mu_{in}\}_{n=1}^{\infty},
\{\tilde\lambda_{in}\}_{n=1}^{\infty},$ where $i=\overline{1,m}.$
Suppose that there exist $M_i,M_i^{*},N_i,N_i^{*}>0,$
$i=\overline{1,m}$ such that $\phi_i(\xi_i)\le M_i^{*}\xi_i,$ for
all $\xi_i\ge M_i$ and $\varphi_i(\zeta_i)\le N_i^{*}\zeta_i$ for
all $\zeta_i\ge N_i,$ where $i=\overline{1,m}.$ Then the following
holds for any $x,y\in K$ and for any $i=\overline{1,m},$
\begin{eqnarray}
\label{inequalityforInxIny}\|I_i^nx-I_i^ny\|&\le&
(1+\tilde\mu_{in}N_i^{*})\|x-y\|+\tilde\mu_{in}\varphi_i(N_i)+\tilde\lambda_{in},\\
\label{inequalityforTnxTny}\|T_i^nx-T_i^ny\|&\le&
(1+\mu_{in}M_i^{*})(1+\tilde\mu_{in}N_i^{*})\|x-y\|+\tilde\mu_{in}(1+\mu_{in}M_i^{*})\varphi_i(N_i)\nonumber\\
&&+\tilde\lambda_{in}(1+\mu_{in}M_i^{*})+\mu_{in}\phi_i(M_i)+\lambda_{in},
\end{eqnarray}
\end{proposition}
\begin{proof}
Since $\phi_i, \varphi_i:{\mathbb{R}}^{+}\to{\mathbb{R}}^{+}$ are
the strictly increasing continuous functions, where
$i=\overline{1,m},$ it follows that $\phi_i(\xi_i)\le\phi_i(M_i)$
and $\varphi_i(\zeta_i)\le\phi_i(N_i)$ whenever $\xi_i\le M_i$ and
$\zeta_i\le N_i,$ where $i=\overline{1,m}.$ By hypothesis of
Proposition \ref{extimationofTnxTny}, for all $\xi_i, \zeta_i\ge 0$
and $i=\overline{1,m},$ we then get
\begin{eqnarray}
\label{inqualityforphi}\phi_i(\xi_i)\le \phi_i(M_i)+M_i^{*}\xi_i,\\
\label{inequalityforvarphi}\varphi_i(\zeta_i) \le
\varphi_i(N_i)+N_i^{*}\zeta_i.
\end{eqnarray}
Since $\{T_i\}_{i=1}^m:K\to K$, $\{I_i\}_{i=1}^m:K\to K$ are  a
total asymptotically $I_i-$nonexpansive and a total asymptotically
nonexpansive mappings, respectively, from \eqref{inqualityforphi}
and \eqref{inequalityforvarphi} one gets
\begin{eqnarray*}
\|I_i^nx-I_i^ny\|&\le&
\|x-y\|+\tilde\mu_{in}\varphi_i(\|x-y\|)+\tilde\lambda_{in}\\
&\le&
\|x-y\|+\tilde\mu_{in}(\varphi_i(N_i)+N_i^{*}\|x-y\|)+\tilde\lambda_{in}\\
&=&
(1+\tilde\mu_{in}N_i^{*})\|x-y\|+\tilde\mu_{in}\varphi_i(N_i)+\tilde\lambda_{in}
\end{eqnarray*}
and
\begin{eqnarray*}
\|T_i^nx-T_i^ny\|&\le&
\|I_i^nx-I_i^ny\|+\mu_{in}\phi_i(\|I_i^nx-I_i^ny\|)+\lambda_{in}\\
&\le&
\|I_i^nx-I_i^ny\|+\mu_{in}(\phi_i(M_i)+M_i^{*}\|I_i^nx-I_i^ny\|)+\lambda_{in}\\
&=& (1+\mu_{in}M_i^{*})\|I_i^nx-I_i^ny\|+\mu_{in}\phi_i(M_i)+\lambda_{in}\\
&\le&
(1+\mu_{in}M_i^{*})(1+\tilde\mu_{in}N_i^{*})\|x-y\|+\tilde\mu_{in}(1+\mu_{in}M_i^{*})\varphi_i(N_i)\\
&&+\tilde\lambda_{in}(1+\mu_{in}M_i^{*})+\mu_{in}\phi_i(M_i)+\lambda_{in}
\end{eqnarray*}
\end{proof}

\begin{lemma}\label{limexistsxnminusp}
Let $X$ be a uniformly convex real Banach space and $K$ be a
nonempty closed convex subset of $X.$ Let $\{T_i\}_{i=1}^m:K\to K$
be a finite family of total asymptotically $I_i-$nonexpansive
mappings with sequences $\{\mu_{in}\}_{n=1}^{\infty},
\{\lambda_{in}\}_{n=1}^{\infty},$ where $i=\overline{1,m},$ and
$\{I_i\}_{i=1}^{m}:K\to K$ be a finite family total asymptotically
nonexpansive mappings with sequences
$\{\tilde\mu_{in}\}_{n=1}^{\infty},
\{\tilde\lambda_{in}\}_{n=1}^{\infty},$ where $i=\overline{1,m},$
such that $F:=\bigcap\limits_{i=1}^m\left(F(T_i)\cap
F(I_i)\right)\neq\emptyset.$ Suppose
$\sum\limits_{n=1}^\infty\mu_{in} < \infty$,
$\sum\limits_{n=1}^\infty\lambda_{in} < \infty,$
$\sum\limits_{n=1}^\infty\tilde\mu_{in} < \infty,$
$\sum\limits_{n=1}^\infty\tilde\lambda_{in} < \infty$ for all
$i=\overline{1,m}$ and there exist $M_i,M_i^{*},N_i,N_i^{*}>0,$
$i=\overline{1,m}$ such that $\phi_i(\xi_i)\le M_i^{*}\xi_i,$ for
all $\xi_i\ge M_i$ and $\varphi_i(\zeta_i)\le N_i^{*}\zeta_i$ for
all $\zeta_i\ge N_i,$ where $i=\overline{1,m}.$ If $\{x_n\}$ is the
explicit iterative sequence defined by \eqref{explicitmap} then for
each $p\in F$ the limit $\lim\limits_{n\to\infty}\|x_n-p\|$ exists.
\end{lemma}

\begin{proof}
Since $F\neq\emptyset,$ for any given $p\in F,$ it
follows from \eqref{explicitmap} and \eqref{inequalityforTnxTny}
that
\begin{eqnarray}
\|x_{n+1}-p\| &=& \left\|\left(1-\sum\limits_{i=1}^m\alpha_{in}\right)(x_{n}-p)+\sum\limits_{i=1}^m\alpha_{in}(T_i^ny_n-p)\right\|\nonumber\\
&\le& \left(1-\sum\limits_{i=1}^m\alpha_{in}\right)\|x_{n}-p\|+\sum\limits_{i=1}^m\alpha_{in}\|T_i^ny_n-p\|\nonumber\\
&\le& \left(1-\sum\limits_{i=1}^m\alpha_{in}\right)\|x_{n}-p\|\nonumber\\
&&+\sum\limits_{i=1}^m\alpha_{in}
(1+\mu_{in}M_i^{*})(1+\tilde\mu_{in}N_i^{*})\|y_n-p\|\nonumber\\
&&+\sum\limits_{i=1}^m\left(\alpha_{in}\tilde\mu_{in}(1+\mu_{in}M_i^{*})\varphi_i(N_i)
+\alpha_{in}\tilde\lambda_{in}(1+\mu_{in}M_i^{*})\right)\nonumber\\
&&+\sum\limits_{i=1}^m\left(\alpha_{in}\mu_{in}\phi_i(M_i)+\alpha_{in}\lambda_{in}\right)\label{estimationofxnandp}.
\end{eqnarray} Again from \eqref{explicitmap} and \eqref{inequalityforInxIny} we derive that
\begin{eqnarray}
\|y_n-p\| &=& \left\|\left(1-\sum\limits_{i=1}^m\beta_{in}\right)(x_n-p)+\sum\limits_{i=1}^m\beta_{in}(I_i^nx_n-p)\right\|\nonumber\\
&\le& \left(1-\sum\limits_{i=1}^m\beta_{in}\right)\|x_n-p\|+\sum\limits_{i=1}^m\beta_{in}\|I_i^nx_n-p\|\nonumber\\
&=&\left(1+\sum\limits_{i=1}^m\tilde\mu_{in}\beta_{in}N_i^{*}\right)\|x_n-p\|\nonumber\\
&&+\sum\limits_{i=1}^m\left(\tilde\mu_{in}\beta_{in}\varphi_i(N_i)
+\tilde\lambda_{in}\beta_{in}\right)\label{estimationofynandp}
\end{eqnarray}
Then from \eqref{estimationofxnandp} and \eqref{estimationofynandp} one finds
\begin{eqnarray}
\|x_{n+1}-p\|&\le& (1+b_n)\|x_n-p\|+c_n\label{inequalityforxnminusp}
\end{eqnarray}
here
\begin{eqnarray*}
b_n&=&\sum\limits_{i=1}^m\left(\mu_{in}\alpha_{in}M_i^{*}+\tilde\mu_{in}\alpha_{in}N_i^{*}
+\alpha_{in}\sum\limits_{i=1}^m\tilde\mu_{in}\beta_{in}N_i^{*}\right)\\
&&+\sum\limits_{i=1}^m\mu_{in}\tilde\mu_{in}\alpha_{in}M_i^{*}N_i^{*}
+\sum\limits_{i=1}^m\mu_{in}\alpha_{in}M_i^{*}\cdot\sum\limits_{i=1}^m\tilde\mu_{in}\beta_{in}N_i^{*}\\
&&+\sum\limits_{i=1}^m\tilde\mu_{in}\alpha_{in}N_i^{*}\cdot\sum\limits_{i=1}^m\tilde\mu_{in}\beta_{in}N_i^{*},\\
c_n&=&\sum\limits_{i=1}^m(\tilde\mu_{in}\beta_{in}\varphi_i(N_i)+\tilde\lambda_{in}\beta_{in})
\cdot\sum\limits_{i=1}^m\alpha_{in}(1+\mu_{in}M_i^{*})(1+\tilde\mu_{in}N_i^{*})\\
&&+\sum\limits_{i=1}^m\left(\tilde\mu_{in}\alpha_{in}(1+\mu_{in}M_i^{*})\varphi_i(N_i)
+\tilde\lambda_{in}\alpha_{in}(1+\mu_{in}M_i^{*})\right)\\
&&+\sum\limits_{i=1}^m\left(\mu_{in}\alpha_{in}\phi_i(M_i)+\lambda_{in}\alpha_{in}\right)\end{eqnarray*}
Denoting $a_n=\|x_{n}-p\|$ in \eqref{inequalityforxnminusp} one gets
\begin{eqnarray*}
a_{n+1}\le(1+b_n)a_n+c_n.
\end{eqnarray*}
Since $\sum\limits_{n=1}^\infty b_n<\infty$ and
$\sum\limits_{n=1}^\infty c_n<\infty,$ it follows from Lemma
\ref{convergean} the existence of the limit
$\lim\limits_{n\to\infty}a_n$. This means the limit
\begin{eqnarray}\label{limofxnminusp}
\lim\limits_{n\to\infty}\|x_n-p\|=d
\end{eqnarray}
exists, where $d\ge0$ is a constant. This completes the proof.
\end{proof}

Now we prove the following result.

\begin{theorem}\label{criteria}
Let $X$ be a uniformly convex real Banach space and $K$ be a
nonempty closed convex subset of $X.$ Let $\{T_i\}_{i=1}^m:K\to K$
be a finite family of total asymptotically $I_i-$nonexpansive
continuous mappings with  sequences $\{\mu_{in}\}_{n=1}^{\infty},
\{\lambda_{in}\}_{n=1}^{\infty},$ where $i=\overline{1,m},$ and
$\{I_i\}_{i=1}^{m}:K\to K$ be a finite family total asymptotically
nonexpansive continuous mappings with sequences
$\{\tilde\mu_{in}\}_{n=1}^{\infty},
\{\tilde\lambda_{in}\}_{n=1}^{\infty},$ where $i=\overline{1,m},$
such that $F:=\bigcap\limits_{i=1}^m\left(F(T_i)\cap
F(I_i)\right)\neq\emptyset.$ Suppose
$\sum\limits_{n=1}^\infty\mu_{in} < \infty$,
$\sum\limits_{n=1}^\infty\lambda_{in} < \infty,$
$\sum\limits_{n=1}^\infty\tilde\mu_{in} < \infty,$
$\sum\limits_{n=1}^\infty\tilde\lambda_{in} < \infty$ for all
$i=\overline{1,m}$ and there exist $M_i,M_i^{*},N_i,N_i^{*}>0,$
$i=\overline{1,m}$ such that $\phi_i(\xi_i)\le M_i^{*}\xi_i,$ for
all $\xi_i\ge M_i$ and $\varphi_i(\zeta_i)\le N_i^{*}\zeta_i$ for
all $\zeta_i\ge N_i,$ where $i=\overline{1,m}.$ Then the explicit
iterative sequence $\{x_n\}$ defined by \eqref{explicitmap}
converges strongly to a common fixed point in $F$ if and only if
\begin{eqnarray}\label{conditionforcriteria}
\liminf\limits_{n\to\infty}d(x_n,F)=0.
\end{eqnarray}
\end{theorem}

\begin{proof}
The necessity of  condition \eqref{conditionforcriteria}  is
obvious. Let us proof the sufficiency part of Theorem.

Since $\{T_i\}_{i=1}^m,\{I_i\}_{i=1}^{m}:K\to K$ are continuous mappings, the sets $F(T_i)$ and
$F(I_i)$ are closed. Hence $F=\bigcap\limits_{i=1}^m\left(F(T_i)\cap F(I_i)\right)$ is a nonempty closed set.

For any given $p\in F,$ we have (see  \eqref{inequalityforxnminusp})
\begin{eqnarray}\label{xnminusp}
\|x_{n+1}-p\|\le\left(1+b_n\right)\|x_{n}-p\|+c_n,
\end{eqnarray}
Hence, one finds
\begin{eqnarray}\label{dxnF}
d(x_{n+1},F) \le \left(1+b_n\right)d(x_{n},F)+c_n
\end{eqnarray}
From \eqref{dxnF} due to Lemma \ref{convergean} we obtain the
existence of the limit $\lim\limits_{n\to\infty}d(x_n,F)$. By
condition \eqref{conditionforcriteria}, one gets
\begin{eqnarray*}
\lim\limits_{n\to\infty}d(x_n,F)=
\liminf\limits_{n\to\infty}d(x_n,F)=0.
\end{eqnarray*}

Let us prove that the sequence $\{x_n\}$ converges strongly to a common fixed
point in $F.$ We first show that $\{x_n\}$ is Cauchy
sequence in $X.$ In fact, due to $1+t\le \exp(t)$ for all $t>0,$ and
from \eqref{xnminusp}, we obtain
\begin{eqnarray}\label{xnminuspandexp}
\|x_{n+1}-p\| &\le& \exp(b_n)(\|x_{n}-p\|+c_n).
\end{eqnarray}
Thus, for any positive integers $m,n,$ from \eqref{xnminuspandexp}
with $\sum\limits_{n=1}^\infty b_n<\infty,$
$\sum\limits_{n=1}^\infty c_n<\infty,$  we find
\begin{eqnarray*}
\|x_{n+m}-p\| &\le& \exp(b_{n+m-1})(\|x_{n+m-1}-p\|+c_{n+m-1})\\
&\le& \exp(b_{n+m-1}+b_{n+m-2})(\|x_{n+m-2}-p\|+c_{n+m-1}+c_{n+m-2})\\
&\le& \cdots \\
&\le& \exp\left(\sum\limits_{i=n}^{n+m-1} b_i\right)\left(\|x_n-p\|+\sum\limits_{i=n}^{n+m-1} c_i\right)\\
&\le& \exp\left(\sum\limits_{i=n}^{\infty}
b_i\right)\left(\|x_n-p\|+\sum\limits_{i=n}^\infty c_i\right).
\end{eqnarray*}
Therefore we get
\begin{eqnarray}\label{xnplusmp}
\|x_{n+m}-x_n\|&\le&\|x_{n+m}-p\|+\|x_n-p\|\nonumber\\
&\le&\left(1+\exp\left(\sum\limits_{i=n}^{\infty}
b_i\right)\right)\|x_n-p\|+\exp\left(\sum\limits_{i=n}^{\infty}
b_i\right)\sum\limits_{i=n}^\infty c_i\nonumber\\
&\le& W\left(\|x_n-p\|+\sum\limits_{i=n}^\infty c_i\right)
\end{eqnarray}
for all $p\in F$, where $0<W-1=\exp\left(\sum\limits_{i=n}^{\infty}
b_i\right)<\infty.$ Taking infimum over $p\in F$ in \eqref{xnplusmp}
gives
\begin{eqnarray}\label{takeninfumumoverp}
\|x_{n+m}-x_n\|&\le& W\left(d(x_n,F)+\sum\limits_{i=n}^\infty
c_i\right)
\end{eqnarray}

Since $\lim\limits_{n\to\infty}d(x_n,F)=0,$ and
$\sum\limits_{i=1}^\infty c_i<\infty,$ given $\varepsilon>0$ there
exists an integer $N_0>0$ such that for all $n>N_0$ we have
$d(x_n,F)<\cfrac{\varepsilon}{2W}$  and $\sum\limits_{i=n}^\infty
c_i<\cfrac{\varepsilon}{2W}.$ Consequently, for all integers $n\ge
N_0$  and $m\ge1$ and from \eqref{takeninfumumoverp} we derive
\begin{eqnarray*}
\|x_{n+m}-x_n\| &\le& \varepsilon.
\end{eqnarray*}
which means that $\{x_n\}$ is Cauchy sequence in $X,$ and since $X$
is complete there exists $x^{*}\in X$ such that the sequence
$\{x_n\}$ converges strongly to $x^{*}.$

Now we show that $x^{*}$ is a common fixed
point in $F.$ Suppose for contradiction that
$x^{*}\notin F.$ Since $F$ is closed subset of $X,$ we have that
$d(x^{*},F)>0.$ However, for all $p\in F,$ we have
$$\|x^{*}-p\|\le\|x_n-x^{*}\|+\|x_n-p\|.$$ This implies that
$$d(x^{*},F)\le\|x_n-x^{*}\|+d(x_n,F),$$ so that as $n\to\infty$ we
obtain $d(x^{*},F)=0$ which contradicts $d(x^{*},F)>0.$ Hence,
$x^{*}$ is a common fixed
point in $F.$ This proves the
required assertion.
\end{proof}

To formulate and prove the main result, we need one more an
auxiliary result.

\begin{proposition}\label{xnTxn&xnIxn}
Let $X$ be a uniformly convex real Banach space and $K$ be a
nonempty closed convex subset of $X.$ Let $\{T_i\}_{i=1}^m:K\to K$
be a finite family of total asymptotically $I_i-$nonexpansive
continuous mappings with  sequences $\{\mu_{in}\}_{n=1}^{\infty},
\{\lambda_{in}\}_{n=1}^{\infty},$ where $i=\overline{1,m},$ and
$\{I_i\}_{i=1}^{m}:K\to K$ be a finite family total asymptotically
nonexpansive continuous mappings with sequences
$\{\tilde\mu_{in}\}_{n=1}^{\infty},
\{\tilde\lambda_{in}\}_{n=1}^{\infty},$ where $i=\overline{1,m},$
such that $F:=\bigcap\limits_{i=1}^m\left(F(T_i)\cap
F(I_i)\right)\neq\emptyset.$ Suppose
$\sum\limits_{n=1}^\infty\mu_{in} < \infty$,
$\sum\limits_{n=1}^\infty\lambda_{in} < \infty,$
$\sum\limits_{n=1}^\infty\tilde\mu_{in} < \infty,$
$\sum\limits_{n=1}^\infty\tilde\lambda_{in} < \infty$ for all
$i=\overline{1,m},$ and $\{\alpha_{jn}\}_{n=1}^{\infty},
\{\beta_{jn}\}_{n=1}^{\infty}$ are sequences  with
$\{\alpha_{jn}\}_{n=1}^{\infty}\subset[\alpha_{*},\alpha^{*}]$ and
$\{\beta_{jn}\}_{n=1}^{\infty}\subset[\beta_{*},\beta^{*}],$ for all
$j=\overline{0,m},$ here $0<\alpha_{*}<\alpha^{*}<1,$
$0<\beta_{*}<\beta^{*}<1,$ and there exist
$M_i,M_i^{*},N_i,N_i^{*}>0,$ $i=\overline{1,m}$ such that
$\phi_i(\xi_i)\le M_i^{*}\xi_i,$ for all $\xi_i\ge M_i$ and
$\varphi_i(\zeta_i)\le N_i^{*}\zeta_i$ for all $\zeta_i\ge N_i,$
where $i=\overline{1,m}.$    Then the explicit iterative sequence
$\{x_n\}$ defined by \eqref{explicitmap}  satisfies the following
\begin{eqnarray}
\lim\limits_{n\to\infty}\|x_n-T_i^nx_n\|&=&0,\label{limofxnTnxn}\\
\lim\limits_{n\to\infty}\|x_n-I_i^nx_n\|&=&0,\label{limofxnInxn}
\end{eqnarray}
for all $i=\overline{1,m}.$
\end{proposition}

\begin{proof} According to Lemma \ref{limexistsxnminusp} for any $p\in F$ we have $\lim\limits_{n\to\infty}\|x_n-p\|=d$. It follows from
\eqref{explicitmap} that
\begin{eqnarray}\label{xnpnd}
\|x_{n+1}-p\|&=& \left\|\alpha_{0n}(x_{n}-p)+\sum\limits_{i=1}^m\alpha_{in}(T_i^ny_n-p)\right\|\to d,
\end{eqnarray} as $n\to\infty.$  By means of $\sum\limits_{n=1}^\infty\mu_{in} < \infty$,
$\sum\limits_{n=1}^\infty\lambda_{in} < \infty,$
$\sum\limits_{n=1}^\infty\tilde\mu_{in} < \infty,$
$\sum\limits_{n=1}^\infty\tilde\lambda_{in} < \infty,$ for all $i=\overline{1,m},$ from
\eqref{estimationofynandp} one yields that
\begin{eqnarray}\label{limsupforynp}
\limsup\limits_{n\to\infty}\|y_n-p\|&\le&\limsup\limits_{n\to\infty}\left[\left(1+\sum\limits_{i=1}^m\tilde\mu_{in}\beta_{in}N_i^{*}\right)\|x_n-p\|\right]\nonumber\\
&&+\limsup\limits_{n\to\infty}\left[\sum\limits_{i=1}^m\left(\tilde\mu_{in}\beta_{in}\varphi_i(N_i)+\tilde\lambda_{in}\beta_{in}\right)\right]\nonumber\\
&=&\limsup\limits_{n\to\infty}\|x_n-p\|=d
\end{eqnarray}
and from \eqref{inequalityforTnxTny}, \eqref{limsupforynp} we have
\begin{eqnarray}\label{Tnynp}
\limsup\limits_{n\to\infty}\|T_i^ny_n-p\| &\le&
\limsup\limits_{n\to\infty}[(1+\mu_{in}M_i^{*})(1+\tilde\mu_{in}N_i^{*})\|y_n-p\|\nonumber\\
&&+\limsup\limits_{n\to\infty}\tilde\mu_{in}(1+\mu_{in}M_i^{*})\varphi_i(N_i)]\nonumber\\
&&+\limsup\limits_{n\to\infty}[\tilde\lambda_{in}(1+\mu_{in}M_i^{*})+\mu_{in}\phi_i(M_i)+\lambda_{in}]\nonumber\\
&\le&d
\end{eqnarray}
for all $i=\overline{1,m}.$ Now using
$$\limsup\limits_{n\to\infty}\|x_{n}-p\|=d$$ with \eqref{Tnynp}
and applying Lemma \ref{generofShu} to \eqref{xnpnd} one finds
\begin{eqnarray}\label{xnTnyn}
\lim\limits_{n\to\infty}\|x_{n}-T_i^ny_n\|=0.
\end{eqnarray}
for all $i=\overline{1,m}.$ Now from \eqref{explicitmap} and \eqref{xnTnyn} we infer that
\begin{eqnarray}\label{xnxnminus1}
\lim\limits_{n\to\infty}\|x_{n+1}-x_{n}\| =
\lim\limits_{n\to\infty}\left\|\sum\limits_{i=1}^m\alpha_{in}\left(T_i^ny_n-x_{n}\right)\right\|=0.
\end{eqnarray}
On the other hand, from \eqref{inequalityforTnxTny} we have
\begin{eqnarray*}
\|x_{n}-p\|&\le&\|x_{n}-T_i^ny_n\|+\|T_i^ny_n-p\|\\
&\le& \|x_{n}-T_i^ny_n\|+(1+\mu_{in}M_i^{*})(1+\tilde\mu_{in}N_i^{*})\|y_n-p\|\\
&&+\tilde\mu_{in}(1+\mu_{in}M_i^{*})\varphi_i(N_i)+\tilde\lambda_{in}(1+\mu_{in}M_i^{*})+\mu_{in}\phi_i(M_i)+\lambda_{in}
\end{eqnarray*}
which implies
\begin{eqnarray*}
\|x_{n}-p\|-\|x_{n}-T_i^ny_n\|&\le&(1+\mu_{in}M_i^{*})(1+\tilde\mu_{in}N_i^{*})\|y_n-p\|\\
&&+\tilde\mu_{in}(1+\mu_{in}M_i^{*})\varphi_i(N_i)+\tilde\lambda_{in}(1+\mu_{in}M_i^{*})\\
&&+\mu_{in}\phi_i(M_i)+\lambda_{in}
\end{eqnarray*}
The last inequality with \eqref{limofxnminusp}, \eqref{xnTnyn} yields
\begin{eqnarray}\label{liminfynp}
\liminf\limits_{n\to\infty}\|y_n-p\|=d
\end{eqnarray}
Combining \eqref{liminfynp} with \eqref{limsupforynp} we get
\begin{eqnarray}\label{ynpd}
\lim\limits_{n\to\infty}\|y_n-p\|=d
\end{eqnarray}
Again from \eqref{explicitmap} we can see that
\begin{eqnarray}\label{ynptod}
\|y_n-p\|=\left\|\beta_{0n}(x_n-p)+\sum\limits_{i=1}^m\beta_{in}(I_i^nx_n-p)\right\|\to d, \ \ \
n\to\infty.
\end{eqnarray}
From \eqref{inequalityforInxIny} and \eqref{limofxnminusp}  one
finds
\begin{eqnarray*}
\limsup\limits_{n\to\infty}\|I_i^nx_n-p\|&\le&\limsup\limits_{n\to\infty}((1+\tilde\mu_{in}N_i^{*})\|x_n-p\|+\tilde\mu_{in}\varphi_i(N_i)
+\tilde\lambda_{in})=d
\end{eqnarray*}
for all $i=\overline{1,m}.$
Now applying Lemma \ref{generofShu} to \eqref{ynptod} we obtain
\begin{eqnarray}\label{xnInxn}
\lim\limits_{n\to\infty}\|x_n-I_i^nx_n\|=0
\end{eqnarray} for all $i=\overline{1,m}.$
We then have
\begin{eqnarray}\label{limxnyn}
\lim\limits_{n\to\infty}\|y_n-x_n\|=\lim\limits_{n\to\infty}\left\|\sum\limits_{i=1}^m\beta_{in}(I_i^nx_n-x_n)\right\|=0
\end{eqnarray}
Consider
\begin{eqnarray*}
\|x_n-T_i^nx_n\| &\le&
\|x_{n}-T_i^ny_n\|+\|T_i^ny_n-T_i^nx_n\|\\
&\le& \|x_{n}-T_i^ny_n\|+(1+\mu_{in}M_i^{*})(1+\tilde\mu_{in}N_i^{*})\|y_n-x_n\|\\
&&+(1+\mu_{in}M_i^{*})(\tilde\mu_{in}\varphi_i(N_i)+\tilde\lambda_{in})+\mu_{in}\phi_i(M_i)+\lambda_{in},
\end{eqnarray*} for all $i=\overline{1,m}.$
Then from \eqref{xnTnyn} and \eqref{limxnyn} we get
\begin{eqnarray*}\label{xnTnxn}
\lim\limits_{n\to\infty}\|x_n-T_i^nx_n\|=0,
\end{eqnarray*}
for all $i=\overline{1,m}.$
\end{proof}

Now we are ready to formulate a main result concerning strong
convergence of the sequence $\{x_n\}$.

\begin{theorem}\label{strongconvergence}
Let $X$ be a uniformly convex real Banach space and $K$ be a
nonempty closed convex subset of $X.$ Let $\{T_i\}_{i=1}^m:K\to K$
be a finite family of total asymptotically $I_i-$nonexpansive
continuous mappings with sequences $\{\mu_{in}\}_{n=1}^{\infty},
\{\lambda_{in}\}_{n=1}^{\infty},$ where $i=\overline{1,m},$ and
$\{I_i\}_{i=1}^{m}:K\to K$ be a finite family total asymptotically
nonexpansive continuous mappings with sequences
$\{\tilde\mu_{in}\}_{n=1}^{\infty},
\{\tilde\lambda_{in}\}_{n=1}^{\infty},$ where $i=\overline{1,m},$
such that $F:=\bigcap\limits_{i=1}^m\left(F(T_i)\cap
F(I_i)\right)\neq\emptyset.$ Suppose
$\sum\limits_{n=1}^\infty\mu_{in} < \infty$,
$\sum\limits_{n=1}^\infty\lambda_{in} < \infty,$
$\sum\limits_{n=1}^\infty\tilde\mu_{in} < \infty,$
$\sum\limits_{n=1}^\infty\tilde\lambda_{in} < \infty$ for all
$i=\overline{1,m},$ and $\{\alpha_{jn}\}_{n=1}^{\infty},
\{\beta_{jn}\}_{n=1}^{\infty}$ are sequences  with
$\{\alpha_{jn}\}_{n=1}^{\infty}\subset[\alpha_{*},\alpha^{*}]$ and
$\{\beta_{jn}\}_{n=1}^{\infty}\subset[\beta_{*},\beta^{*}],$ for all
$j=\overline{0,m},$ here $0<\alpha_{*}<\alpha^{*}<1,$
$0<\beta_{*}<\beta^{*}<1,$ and there exist
$M_i,M_i^{*},N_i,N_i^{*}>0,$ $i=\overline{1,m}$ such that
$\phi_i(\xi_i)\le M_i^{*}\xi_i,$ for all $\xi_i\ge M_i$ and
$\varphi_i(\zeta_i)\le N_i^{*}\zeta_i$ for all $\zeta_i\ge N_i,$
where $i=\overline{1,m}.$  If at least one mapping of the mappings
$\{T_i\}_{i=1}^m$ and $\{I_i\}_{i=1}^{m}$ is compact, then the
explicitly iterative sequence $\{x_n\}$ defined by
\eqref{explicitmap} converges strongly to a common fixed point of
$\{T_i\}_{i=1}^m$ and $\{I_i\}_{i=1}^{m}.$
\end{theorem}

\begin{proof}
Without any loss of generality, we may assume that $T_1$ is compact.
This means that there exists a subsequence $\{T_1^{n_k}x_{n_k}\}_{k=1}^\infty$ of
$\{T_1^nx_n\}_{n=1}^\infty$ such that $\{T_1^{n_k}x_{n_k}\}_{k=1}^\infty$ converges strongly to
$x^{*}\in K.$ Then from \eqref{limofxnTnxn} we have that
$\{x_{n_k}\}_{k=1}^\infty$ converges strongly to $x^{*}.$ Also from \eqref{limofxnTnxn}, we obtain   $\{T_i^{n_k}x_{n_k}\}_{k=1}^\infty$ converges strongly to $x^{*},$ for all $i=\overline{2,m}.$  Since $\{T_i\}_{i=1}^m$ are continuous mappings, so  $\{T_i^{n_k+1}x_{n_k}\}_{k=1}^\infty$ converges strongly to $T_ix^{*},$ for all $i=\overline{1,m}.$ On the other hand, from \eqref{limofxnInxn} and
continuousness of $\{I_i\}_{i=1}^m$ we obtain that $\{I_i^{n_k}x_{n_k}\}_{k=1}^\infty$ converges
strongly to $x^{*}$ and $\{I_i^{n_k+1}x_{n_k}\}_{k=1}^\infty$ converges strongly to
$I_ix^{*},$ for all $i=\overline{1,m}.$ Due to \eqref{xnxnminus1}, $\{\|x_{n_k+1}-x_{n_k}\|\}$ converges to $0,$ as $k\to\infty.$
 Then, $\{x_{n_k+1}\}_{k=1}^\infty$ converges strongly to $x^{*}$ and moreover, \eqref{inequalityforTnxTny} and \eqref{inequalityforInxIny} imply
that $\{\|T_i^{n_k+1}x_{n_k+1}-T_i^{n_k+1}x_{n_k}\|\}$ and
$\{\|I_i^{n_k+1}x_{n_k+1}-I_i^{n_k+1}x_{n_k}\|\}$ converge to $0,$  as $k\to\infty,$ for all $i=\overline{1,m}.$ From \eqref{limofxnTnxn}, \eqref{limofxnInxn} it yields that $\|x_{n_k+1}-T_i^{n_k+1}x_{n_k+1}\|$ and $\|x_{n_k+1}-I_i^{n_k+1}x_{n_k+1}\|$ converge to $0$ as $k\to\infty,$ for all $i=\overline{1,m}.$
Observe that
\begin{eqnarray*}
\|x^{*}-T_ix^{*}\|&\le&\|x^{*}-x_{n_k+1}\|+\|x_{n_k+1}-T_i^{n_k+1}x_{n_k+1}\|\\
&&+\|T_i^{n_k+1}x_{n_k+1}-T_i^{n_k+1}x_{n_k}\|+\|T_i^{n_k+1}x_{n_k}-T_ix^{*}\|,\\
\|x^{*}-I_ix^{*}\|&\le&\|x^{*}-x_{n_k+1}\|+\|x_{n_k+1}-I_i^{n_k+1}x_{n_k+1}\|\\
&&+\|I_i^{n_k+1}x_{n_k+1}-I_i^{n_k+1}x_{n_k}\|+\|I_i^{n_k+1}x_{n_k}-I_ix^{*}\|,\\
\end{eqnarray*}
for all $i=\overline{1,m}.$ Taking limit as $k\to\infty$ we have that $x^{*}=T_ix^{*}$ and $x^{*}=I_ix^{*},$ for all $i=\overline{1,m},$ which means $x^{*}\in F.$
However, by Lemma \ref{limexistsxnminusp}, the limit $\lim\limits_{n\to\infty}\|x_n-x^{*}\|$ exists then
$$\lim\limits_{n\to\infty}\|x_n-x^{*}\|=\lim\limits_{n_k\to\infty}\|x_{n_k}-x^{*}\|=0,$$
which means $\{x_n\}$ converges strongly to $x^{*}\in F.$ This completes the
proof.
\end{proof}

{\bf Remark.} If one has that all $I_i$ are identity mappings, then
the obtained results recover and correctly prove the main result of
\cite{[ChidumeOfoedu2009]}.

{\bf Remark.} Suppose we are given two family $\{T_i\}_{i=1}^m:K\to
K$ and $\{S_i\}_{i=1}^m:K\to K$  of total asymptotically
nonexpansive continuous mappings such that
$\bigcap\limits_{i=1}^m\left(F(T_i)\cap F(S_i)\right)\neq\emptyset.$
Define the following explicit iterative process:
\begin{eqnarray}\label{explicitmap111}
\left\{ \begin{array}{ccc}
          x_0\in K,\\
          x_{n+1} =  \alpha_{0n}x_{n}+\sum\limits_{i=1}^m\alpha_{in} T_i^ny_n +\alpha_{m+1,n}u_n\\
          y_n =  \beta_{0n}x_n+\sum\limits_{i=1}^m\beta_{in} S_i^nx_n+\beta_{m+1,n}v_n.
        \end{array}\right.
\end{eqnarray}
such that $\sum\limits_{j=0}^{m+1}\alpha_{jn}=1$ and
$\sum\limits_{j=0}^{m+1}\beta_{jn}=1.$

Under suitable conditions, by the same argument and methods used
above one can prove, with  either little mirror or no modifications,
the strong convergence of the explicit iterative process defined by
\eqref{explicitmap111} to a common fixed point of the given
families.

 {\bf Remark.} Let $\{T_i\}_{i=1}^m:K\to K$ be a finite
family of total asymptotically nonexpansive continuous mappings with
sequences $\{\mu_{in}\}_{n=1}^{\infty},
\{\lambda_{in}\}_{n=1}^{\infty},$ where $i=\overline{1,m}$. It is
clear for each operator $T_i$ one has
\begin{equation}\label{t-i}
\|T^n_ix-T^n_iy\|\leq \|T^n_ix-T^n_iy\|+\mu_{in}\|T^n_ix-T^n_iy\|,
\end{equation}
this means that $T_i$ is total asymptotically $T_i$-nonexpansive mappings with sequence $\{\mu_{in}\}_{n=1}^{\infty}$ and the function $\phi(\lambda)=\lambda$. Hence, our iteration scheme can be
written as follows
\begin{eqnarray}\label{1explicitmap}
\left\{ \begin{array}{ccc}
          x_0\in K,\\
          x_{n+1} =  \alpha_{0n}x_{n}+\sum\limits_{i=1}^m\alpha_{in} T_i^ny_n \\
          y_n =  \beta_{0n}x_n+\sum\limits_{i=1}^m\beta_{in} T_i^nx_n.
        \end{array}\right.
\end{eqnarray}
where $\{\alpha_{jn}\}_{n=1}^\infty,$ $\{\beta_{jn}\}_{n=1}^\infty$ in
$(0,1),$ ($j=\overline{0,m})$ with   $\sum\limits_{j=0}^m\alpha_{jn}=1$, $\sum\limits_{j=0}^m\beta_{jn}=1.$

The defined scheme is a new iterative method generalizing one given
in \cite{[ChidumeOfoedu2009]}. So, according our main results for
the defined sequence $\{x_n\}$ (see \eqref{1explicitmap}) we obtain
strong convergence theorems. On the other hand, playing with numbers
$\{\alpha_{jn}\}_{n=1}^\infty,$ $\{\beta_{jn}\}_{n=1}^\infty$ and by
means of the defined method one may introduce lots of different
schemes. All of the them strongly converges to a common fixed point
of $\{T_i\}_{i=1}^m$.   Moreover, the recursion formula
\eqref{1explicitmap} is much simpler than the others studied earlier
for this problem \cite{[AlberChidumeZegeye]},
\cite{[ChidumeOfoedu]}, \cite{Gu},
\cite{T},\cite{[Sahu]},\cite{AEL},\cite{MS},\cite{GC},\cite{ZCK}.
Therefore, all presented results here generalize, unify and extend
the corresponding main results of the mentioned papers. Note that
one can consider the method \eqref{explicitmap} with errors, and all
the theorems could be carry over for such iteration scheme as well
with little or no modifications.

We stress that  all the theorems of this paper carry over to the
class of total asymptotically quasi $I$-nonexpansive mappings (see
\cite{[ChidumeOfoedu2009]}), \cite{ZH} with little or no
modifications.

\section*{Acknowledgments} A part of this work was done at the Abdus Salam International
Center  for Theoretical Physics (ICTP), Trieste, Italy. The first
named author (FM) thanks the ICTP for providing financial support
during his visit as a  Junior Associate at the
centre. The authors also acknowledge the Malaysian
Ministry of Science, Technology and Innovation Grant
01-01-08-SF0079.

\end{document}